\documentclass[a4paper]{article}

\usepackage[margin=2.5cm]{geometry}
\usepackage{amsmath,amsthm,amsfonts,amssymb, tikz, array, tabularx, url}
\usepackage{comment}
\usetikzlibrary{calc}

\newcommand{\ex}{\text{ex}}

\newtheorem{thm}{Theorem}[section]
\newtheorem{lemma}[thm]{Lemma}

\theoremstyle{definition}

\newtheorem{defn}[thm]{Definition}

\numberwithin{equation}{section}
\numberwithin{figure}{section}

\title{Exact generalized Tur\'an number for $K_3$ versus suspension of $P_4$}

\author{
	Sayan Mukherjee\footnote{Blueqat Research, Tokyo, Japan. Email:{\tt sayan@blueqat.com}}
	\footnote{Department of Physics, The University of Tokyo, Japan.}
}

\begin{document}
	\maketitle
 \begin{abstract}
	\noindent
	Let $P_4$ denote the path graph on $4$ vertices.
	The suspension of $P_4$, denoted by $\widehat P_4$, is the graph obtained via adding an extra vertex and joining it to all four vertices of $P_4$.
	In this note, we demonstrate that for $n\ge 8$, the maximum number of triangles in any $n$-vertex graph not containing $\widehat P_4$ is $\left\lfloor n^2/8\right\rfloor$.
	Our method uses simple induction along with computer programming to prove a base case of the induction hypothesis.
	\\[2mm]
\end{abstract}
	
	\section{Introduction}
	
	The generalized Tur\'an number $\ex(n, T, H)$ is defined as the maximum number of copies of $T$ in an $n$-vertex graph not containing $H$ as a (not necessarily induced) subgraph.
	When $T=K_2$, this is the Tur\'an number $\ex(n,H)$ of the graph.
	The first systematic study of $\ex(n, T, H)$ for $T\neq K_2$ was carried out by Alon and Shikhelman~\cite{manyTcopies-alon2016}.
	
	In more recent years, several researchers have studied the asymptotic behavior of $\ex(n, K_3, H)$ for the case $T=K_3$ (see, for example \cite{trianglesC5Free-hungary2019,trianglesC5Free-hungary2018,trianglesOddCycleFree-hungary2012}).
	It is known that when $\chi(H)>3$, $\ex(n,K_3,H)\sim \binom{\chi(H)-1}{3}/{(\chi(H)-1)^2} \cdot n^2$, where $\chi(H)$ denotes the chromatic number of $H$~\cite{manyTcopies-alon2016,sharpResults-ma2020}.
	Alon and Shikhelman~\cite{manyTcopies-alon2016} extensively study the case when $\chi(H)=2$.
	
	Mubayi and the author~\cite{suspensionFree2023} initiated the study of $\ex(n, K_3, H)$ for a simple family of graphs $H$ with $\chi(H)=3$.
	For any graph $G$, they denoted the suspension $\widehat G$ as the graph obtained from $G$ by adding a new vertex $v$ and joining it with all vertices of $G$.
	They proceeded to analyze the asymptotic behavior of $\ex(n,K_3,\widehat G)$ for different bipartite graphs $G$.
	
	One of the several bipartite graphs they consider is the path $P_4$ on four vertices. It was shown that for any $n\ge 4$,
	\begin{equation}
		\frac{n^2}{8}-O(1)\le \ex(n, K_3, \widehat P_4) < \frac{n^2}{8}+3n.
	\end{equation}
	
	An exact result for sufficiently large $n$ was given by Gerbner~\cite{hatP3Free-gerbner2022} using the technique of progressive induction. In particular, they prove that for a number $K\le 1575$ and $n\ge 525+4K$,
	\begin{equation}
		\label{eq:maineq}
		\ex(n,K_3,\widehat P_4) = \left\lfloor n^2/8\right\rfloor.
	\end{equation}

	They mention that a proof of the upper bound of (\ref{eq:maineq}) for $n=8,9,10,11$ together with induction would suffice to prove (\ref{eq:maineq}) for every $n\ge 8$.	
	In this note, we leverage this idea to determine the exact value of $\ex(n, K_3, \widehat P_4)$ for \emph{every} $n\ge 4$, thus closing the gap in the literature for this extremal problem.
    
	\begin{thm}
		\label{thm:mainthm}
		For $n\ge 8$, $\ex(n, K_3, \widehat P_4) = \left\lfloor n^2/8\right\rfloor$. For $n=4,5,6,7$ the values of $\ex(n, K_3,\widehat P_4)$ are $4,4,5,8$ respectively.
	\end{thm}
	
	The lower bound constructions for Theorem~\ref{thm:mainthm} are different for the cases $n\in\{4,5,6,7\}$ and $n\ge 8$.
	
	Figure~\ref{fig:smallExamples} illustrates graphs on $n$ vertices for $n\in\{4,5,6,7\}$ that achieve the maximum number of triangles.
	In fact, we shall see later in Section~\ref{subsec:brutalforce} that these constructions are unique up to isomorphism.
	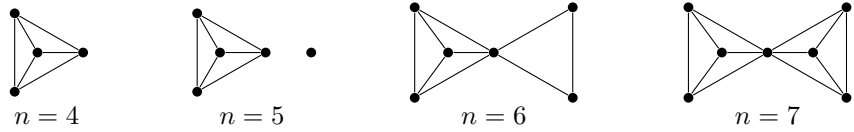
\begin{figure}[ht]
		\centering
		\begin{tikzpicture}[scale = 0.6]
			\tikzstyle{vertex}=[circle,fill=black,minimum size=2pt,inner sep=1.3pt]
			
			\begin{scope}[shift={(0,0)}]
				\node[vertex] (a) at (0:1) {};
				\node[vertex] (b) at (120:1) {};
				\node[vertex] (c) at (240:1) {};
				\node[vertex] (d) at (0,0) {};
				\draw (a)--(b)--(c)--(d)--(a)--(c);
				\draw (d)--(b);
				\draw (0.2,-1) node [below] {$n=4$};
			\end{scope}
			
			\begin{scope}[shift={(4,0)}]
				\node[vertex] (a) at (0:1) {};
				\node[vertex] (b) at (120:1) {};
				\node[vertex] (c) at (240:1) {};
				\node[vertex] (d) at (0,0) {};
				\node[vertex] (e) at (0:2) {};
				\draw (a)--(b)--(c)--(d)--(a)--(c);
				\draw (d)--(b);
				\draw (0.7,-1) node [below] {$n=5$};
			\end{scope}
			
			\begin{scope}[shift={(10,0)}]
				\node[vertex] (a) at (0:0) {};
				\node[vertex] (b) at (150:2) {};
				\node[vertex] (c) at (210:2) {};
				\node[vertex] (d) at (-1,0) {};
				\node[vertex] (e) at (30:2) {};
				\node[vertex] (f) at (-30:2) {};
				\draw (a)--(e)--(f)--(a)--(b)--(c)--(d)--(a)--(c);
				\draw (d)--(b);
				\draw (0,-1) node [below] {$n=6$};
			\end{scope}
			
			\begin{scope}[shift={(16,0)}]
				\node[vertex] (a) at (0:0) {};
				\node[vertex] (b) at (150:2) {};
				\node[vertex] (c) at (210:2) {};
				\node[vertex] (d) at (-1,0) {};
				\node[vertex] (e) at (30:2) {};
				\node[vertex] (f) at (-30:2) {};
				\node[vertex] (g) at (1,0) {};
				\draw (a)--(e)--(f)--(a)--(b)--(c)--(d)--(a)--(c);
				\draw (d)--(b);
				\draw (g)--(e);
				\draw (g)--(a);
				\draw (g)--(f);
				\draw (0,-1) node [below] {$n=7$};
			\end{scope}
			
		\end{tikzpicture}
		\caption{\label{fig:smallExamples}Graphs on $4,5,6,7$ vertices and $4,4,5,8$ triangles, respectively.}
	\end{figure}

	The general lower bound construction considered in \cite{hatP3Free-gerbner2022,suspensionFree2023} (for $n\ge 8$) was the complete bipartite graph $K_{\lfloor n/2\rfloor, \lceil n/2\rceil}$ with a matching in any of the even parts.
	A short case analysis shows that the total number of triangles in these graphs is given by $\left\lfloor n^2/8\right\rfloor$, hence proving the lower bound in Theorem~\ref{thm:mainthm} for general $n$.
	
	Thus, the main goal of this manuscript is to prove that these lower bounds on $\ex(n,K_3,\widehat P_4)$ are tight.
	
	This work is organized as follows.
	We present some preliminaries in Section~\ref{sec:prelim}.
	Then, we show the upper bound of Theorem~\ref{thm:mainthm} for $n\ge 5$ in Section~\ref{sec:upperbd}.
	Finally, we make some concluding remarks regarding uniqueness of the lower bound constructions in Section~\ref{sec:concl}.
	
	
	\section{Preliminaries}\label{sec:prelim}
	
	Throughout the rest of this paper, we assume without loss of generality that all graphs are \emph{edge-minimal}.
	This implies that every edge of the graphs considered must lie in a triangle, as we can simply delete edges that do not help forming a triangle.
	We also assume that the vertex set of any $n$-vertex graph in the rest of this section is $\{0,\ldots,n-1\}$, and abuse notation to represent a $K_3$ on vertex subset $\{a,b,c\}$ as simply $abc$.
	
	Let $n(G)$, $e(G)$ and $t(G)$ denote the number of vertices, edges and triangles in $G$, respectively.
	
	Now we recall some definitions and state a two important lemmas from \cite{hatP3Free-gerbner2022} and \cite{suspensionFree2023} which are instrumental in our proof.
	
	\begin{defn}[Triangle-connectivity]
		For a graph $G$, two edges $e$ and $e'$ are said to be \emph{triangle-connected} if there is a sequence of triangles $\{T_1,\ldots,T_k\}$ of $G$ such that $e\in T_1$, $e'\in T_k$, and $T_i$ and $T_{i+1}$ share a common edge for every $1\le i \le k-1$.
		A subgraph $H\subseteq G$ is triangle-connected if $e$ and $e'$ are triangle-connected for every edges $e$ and $e'$ of $H$.
	\end{defn}
	
	\begin{defn}[Triangle block]
		A subgraph $H\subseteq G$ is a \emph{triangle block} (or simply a \emph{block}) if it is edge-maximally triangle-connected.
	\end{defn}
	
	By definition, the triangle blocks of any graph $G$ are edge-disjoint.
	
	Let $B_s$ denote the \emph{book graph} on $(s+2)$ vertices, consisting of $s$ triangles all sharing a common edge. Let this common edge be called the \emph{base} of the $B_s$. The following lemma characterizes the triangle blocks of any $\widehat P_4$-free graph $G$.
	
	\begin{lemma}[\cite{suspensionFree2023}, Claim 5.3]
		\label{lem:block-characterization}
		Every triangle block of a $\widehat P_4$-free graph $G$ is isomorphic to a $K_4$ or a $B_s$ for some $s\ge 1$.
	\end{lemma}
	\begin{proof}
		Let $H\subseteq G$ be an arbitrary triangle block.
		If $H$ contains only one or two triangles, it is isomorphic to $B_1$ or $B_2$.
		Suppose $H$ contains at least three triangles.
		Let two of them be $abx_1$ and $abx_2$ (see Figure~\ref{fig:blockp3}).
		
		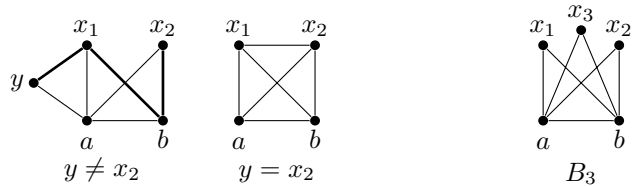
\begin{figure}[ht]
			\centering
			\begin{tikzpicture}
				\tikzstyle{vertex}=[circle,fill=black,minimum size=2pt,inner sep=1.3pt]
				\node[vertex] (a) at (0,0){};
				\node[vertex] (b) at (1,0){};
				\node[vertex] (x1) at (0,1){};
				\node[vertex] (x2) at (1,1){};
				\node[vertex] (y) at (-0.7,0.5){};
				\draw (a)--(x1)--(b)--(x2)--(a)--(b);
				\draw [line width = 1pt] (y)--(x1)--(b)--(x2);
				\draw (x1)--(y)--(a);
				\draw (0,-0.1) node [below] {$a$};
				\draw (b) node [below] {$b$};
				\draw (x1) node [above] {$x_1$};
				\draw (x2) node [above] {$x_2$};
				\draw (y) node [left] {$y$};
				\draw (0.2,-0.65)  node {$y\neq x_2$};
				
				\node[vertex] (ap) at (2,0){};
				\node[vertex] (bp) at (3,0){};
				\node[vertex] (x1p) at (2,1){};
				\node[vertex] (x2p) at (3,1){};
				\draw (ap)--(x1p)--(bp)--(x2p)--(ap)--(bp);
				\draw (x1p)--(x2p);
				\draw (2,-0.1) node [below] {$a$};
				\draw (bp) node [below] {$b$};
				\draw (x1p) node [above] {$x_1$};
				\draw (x2p) node [above] {$x_2$};
				\draw (2.5,-0.7)  node {$y=x_2$};
				
				\node[vertex] (app) at (6,0){};
				\node[vertex] (bpp) at (7,0){};
				\node[vertex] (xpp) at (6,1){};
				\node[vertex] (ypp) at (7,1){};
				\node[vertex] (zpp) at (6.5,1.2){};
				\draw (app)--(bpp)--(xpp)--(app)--(ypp)--(bpp)--(zpp)--(app);
				\draw (6,-0.1) node [below] {$a$};
				\draw (bpp) node [below] {$b$};
				\draw (xpp) node [above] {$x_1$};
				\draw (ypp) node [above] {$x_2$};
				\draw (zpp) node [above] {$x_3$};
				\draw (6.5,-0.7)  node {$B_3$};
			\end{tikzpicture}
			\caption{(left): third triangle on $ax_1$, (right): third triangle on $ab$\label{fig:blockp3}}
		\end{figure}
		
		If another triangle is of the form $ax_1y$ for some $y\in V(H)$, then there are two possible cases. If $y\neq x_2$, then $N_H(a)$ contains the 4-path $x_2bx_1y$, a contradiction. Otherwise if $y=x_2$, then the vertices $a,b,x_1,x_2$ create a $K_4$, and this $K_4$ is a triangle block by itself.
		
		Similarly, if a triangle contained any of the edges $bx_1, ax_2, bx_2$, we would end up with a $K_4$-block, and this block cannot be extended any further.
			
		Therefore all triangles in $H$ would intersect the edge $ab$, implying $H\cong B_s$ for some $s\ge 1$.	
	\end{proof}

	\begin{lemma}[\cite{hatP3Free-gerbner2022}, Section 2]
		\label{lem:mustcontainK4}
		Suppose $G$ is an $n$-vertex $\widehat P_4$-free graph containing no $K_4$. Then, we have
		$t(G)\le \lfloor n^2/8\rfloor$.
	\end{lemma}
	\begin{proof}
		By Lemma~\ref{lem:block-characterization}, all triangle blocks of $G$ are isomorphic to $B_s$ for some $s\ge 1$.
		Let $G'$ be obtained from $G$ by deleting the base edges of each of the books (if $s=1$, delete any arbitrary edge).
		As each triangle of $G$ contains two distinct edges from $G'$, we have $t(G)=e(G')/2$.
		By Mantel's theorem, $e(G')\le \lfloor n^2/4\rfloor$, implying $t(G)\le \frac 12\lfloor n^2/4\rfloor$, i.e. $t(G)\le \lfloor n^2/8\rfloor$.
	\end{proof}

	
	\section{Upper bounds}\label{sec:upperbd}
	
	In order to prove that $\ex(n,K_3,\widehat P_4)\le K$ for some fixed $n$ and $K$, we need to show that any $n$-vertex graph containing at least $K+1$ triangles contains a copy of $\widehat P_4$.

	\subsection{The cases $5\le n\le 8$: brute force}\label{subsec:brutalforce}
	
	While a case-by-case analysis is tractable by hand for $n=5$ for example, we quickly run into several possible configurations while trying to prove $\ex(8,K_3,\widehat P_4)=8$.
	This is where we turn to a computer-generated check.
	For example, to prove that all $8$-vertex graphs with more than $9$ triangles is $\widehat P_4$-free, we can assume that $012$ and $013$ are two triangles in some $8$-vertex graph $G$ containing $9$ triangles.
	Then triangles that have an edge from the set $\{02, 03, 12, 13\}$ and have a node from $\{4,5,6,7\}$ are excluded from $G$ since any of these patterns form a $\widehat P_4$.
	This excludes $16$ triangles.
	Hence the plausible triangles that $G$ may contain other than $012$ and $013$ are $\binom 83 - 18 = 38$ in number.
	We generate $\binom {38}{7}\approx 1.26\times 10^7$ possible graphs, filter out the ones that have exactly $9$ triangles, and check for $\widehat P_4$'s in each of them.
	
	Our program is available at the Github repository in \cite{Mukherjee_Exact_generalized_Tur_an_2023}.
	We run \verb*|triangle_count_parallel.py| for different pairs of $(n, t)$ to figure out both the extremal number and all extremal configurations for $n$-vertex graphs with $t$-triangles.
	The results are compiled in the notebook \verb*|triangle_count.ipynb|.
	Our computation shows that $\text{ex}(n,K_3,\widehat P_4)=4,5,8,8$ for $n=5,6,7,8$, respectively.
	The total computation time required for $(n,t)=(8,9)$ on 7 threads of an \verb|Intel(R) Core(TM) i7-8550U| laptop processor running at 1.80GHz was around 18 minutes.
	
	\subsection{The cases $9\le n \le 11$: identifying $K_4$}
	
	The main idea behind these cases is to follow the steps of the proof in \cite{suspensionFree2023}, Section 5.2.
	
	\begin{thm}
		\label{thm:9-11-theorem}
		Suppose $(n,t)\in\{(9,11), (10,13), (11, 16)\}$, and $G$ is an (edge-minimal) $n$-vertex graph with $t$ triangles.
		Then $G$ must contain a $\widehat P_4$.
	\end{thm}
	\begin{proof}
		For the sake of contradiction, assume that $G$ was $\widehat P_4$-free.
		If $G$ was also $K_4$-free, then by Lemma~\ref{lem:mustcontainK4}, $t(G)\le \lfloor n^2/8\rfloor = 10, 12, 15$ for $n=9,10,11$, contradicting our initial assumption on $t(G)$.
		
		Therefore $G$ must contain a $K_4$.
		Let this $K_4$ be induced by vertex subset $S=\{u_0,u_1,u_2,u_3\}\subset V(G)$.
		Define $X_i := N(u_i) - S$ for $0\le i \le 3$.
		As $G[S]$ is a triangle block, $X_i\cap X_j=\varnothing$ for every $i\neq j$.
		Further, $\sum_{i=0}^3|X_i|\le n-4$.
		Without loss of generality assume $|X_0|\le \cdots\le |X_3|$.
		Now we consider each case separately.
		\begin{itemize}
			\item {\bf Case 1. }${(n,t)=(9,11)}$:
			In this case, $\sum_{i=0}^3|X_i|\le 5$.
			If $|X_1|>0$, by edge-minimality we would have $|X_1|\ge 2$, implying $|X_1|+|X_2|+|X_3|\ge 6$, a contradiction.
			Thus, $|X_0| = |X_1| = 0$, and by a similar argument, $|X_2|\le 2$.
			This means the vertex $u_2$ lies in at most one triangle outside of $G[S]$.
			Let $G'$ be obtained by deleting $\{u_0,u_1,u_2\}$ from $G$. Clearly $n(G')=6$ and $t(G')\ge t(G)-5 = 6$.
			As $\ex(6,K_3,\widehat P_4)=5$ by the discussion in Section~\ref{subsec:brutalforce}, $G'$ has a $\widehat P_4$, a contradiction.
			
			\item {\bf Case 2. }$(n,t)=(10,13)$:
			Here, $\sum_{i=0}^3 |X_i|\le 6$.
			By a similar analysis as before, we can infer that $|X_0|=0$ and $|X_1|\le 2$.
			If $|X_1|=0$, we could consider $G'=G-\{u_0,u_1\}$, which would have $n(G')=8$ and $t(G')=13-4=9$, which would lead us to a $\widehat P_4$ since $\ex(8,K_3,\widehat P_4)=8$ by the calculation in Section~\ref{subsec:brutalforce}.
			Thus, we have $|X_0|=0$, $|X_1|=2$, and hence $|X_2|=|X_3|=2$.
			Now, if we consider $G''=G-S$, we have $n(G'')=6$ and $t(G'') = 13-4-3=6$, again implying that $G''$ has a $\widehat P_4$.
			
			\item {\bf Case 3. }$(n,t)=(11,16)$:
			For this pair of $(n,t)$, we have $\sum_{i=0}^3|X_i|\le 7$, implying $|X_0|=0$ again.
			Since $u_0$ lies in exactly three triangles of $G[S]$, $G'=G-\{u_0\}$ has $n(G')=10$ and $t(G')=13$, leading us to the previous case.
		\end{itemize}
	
	In either of the three cases, we obtain a contradiction, finishing the proof for these cases.
	\end{proof}
	
	\subsection{The case $n\ge 12$: identifying $K_4$}\label{sec:inductionStep}
	
	Now that we have proved $\ex(n,K_3,\widehat P_4) = \lfloor n^2/8\rfloor$ for $8\le n \le 11$, we are now ready to handle the general case using induction on $n$.
    Our proof follows the idea of \cite{hatP3Free-gerbner2022} with a more careful analysis to obtain the desired bound.
	
	\begin{proof}[Proof of Theorem~\ref{thm:mainthm} for $n\ge 12$.]
		Let us assume that $\ex(k,K_3,\widehat P_4)=\lfloor k^2/8\rfloor$ for all $8\le k \le n-1$.
        We note that a simple case analysis leads to
        \begin{equation}
            \label{eq:induction-1-3-4-vertex}
            \begin{aligned}
            \lfloor n^2/8\rfloor - \lfloor (n-1)^2/8\rfloor &\ge \lfloor n/4\rfloor\\ 
            \lfloor n^2/8\rfloor - \lfloor (n-4)^2/8\rfloor &= n-2.
            \end{aligned}
        \end{equation}
		For the sake of contradiction, suppose $G$ is an $n$-vertex $\widehat P_4$-free graph with $t(G)\ge \lfloor n^2/8\rfloor +1$.
        For a subset $U\subset V(G)$, let us denote by $t(U)$ the number of triangles containing at least one vertex from $U$.
        By (\ref{eq:induction-1-3-4-vertex}), we may assume that
        \begin{equation}
            \label{eq:assumption-1-3-4-vertex}
            \begin{aligned}
            |U|=1\implies t(U)&\ge \lfloor n/4\rfloor +1,\\ 
            |U|=4\implies t(U)&\ge n-1.
            \end{aligned}
        \end{equation}
		
        Now, notice that by Lemma~\ref{lem:mustcontainK4}, $G$ must contain a $K_4$.
		As in the previous section, let $S=\{u_0,u_1,u_2,u_3\}$ induce this $K_4$, and denote $X_i=N(u_i)-S$ for $0\le i \le 3$. Again, $|X_i\cap X_j|=\varnothing$ for every $i\neq j$.
        Observe that $t(S)= \sum_{i=0}^3 e(X_i)+4$, and so by (\ref{eq:assumption-1-3-4-vertex}), 
		\[
		\sum_{i=0}^3 e(X_i)\ge n-5.
		\]
		
		On the other hand, since each $X_i$ is $P_4$-free, we have $\sum_{i=0}^3 e(X_i)\le \sum_{i=0}^3 |X_i| \le n-4$.
        Hence,
        \begin{equation}
            \label{eq:e(Xi)-limited}
            \sum_{i=0}^3 e(X_i)\in \{n-5, n-4\}
        \end{equation}
		This implies that $e(X_i)=|X_i|$ for at least {three} $u_i\in S$.
		Assume that $e(X_i)=|X_i|$ for $0\le i \le 2$ and $e(X_3)\in\{|X_3|-1, |X_3|\}$.
        This also means that $G[X_i]$ are vertex-disjoint unions of triangles for $0\le i\le 2$,
        and $X_3$ is a union of triangles and a star on $r$ vertices for some $r\ge 0$.
        Further, (\ref{eq:assumption-1-3-4-vertex}) gives us the bound
        \begin{equation}
            \label{eq:bound-Xi-lower}
            |X_i|\ge \lfloor n/4\rfloor -2 .
        \end{equation}

        We now continue with a more detailed analysis of the neighborhoods of vertices in $G$.
        In what follows, let $x_i$ denote the size of $X_i$.
        For a subset $A\subset V(G)$, let $\mathcal T(A)$ denote the set of triangles in $G[A]$.
        We now consider two cases.
        
        \paragraph{Case 1: $\sum_{i=0}^4x_i=n-5$.}
        
        In this case, note that since $\sum_{i=0}^3 e(X_i) = n-5$, we have $e(X_3)=x_3$.
        Thus, the subgraphs $G[X_i]$ are all disjoint unions of triangles, and there is exactly one vertex $y$ in $V(G)-\bigcup_i X_i \cup S$,
        and thus $3\mid n-5$, implying $n\equiv 2 \mod 3$.
        Moreover, (\ref{eq:bound-Xi-lower}) implies $x_i\ge 3$, and hence $n\ge 17$.

        Now, observe that for $G'=G-\{y\}$,
        \begin{equation}
            \sum_{v\in V(G)}\deg v = \sum_{i=0}^3\sum_{vwz\in \mathcal T(X_i)}(\deg_{G'}v+\deg_{G'}w+\deg_{G'}z) + \sum_{v\in S}\deg v + 2\deg y.
            \label{eq:break-sum-over-V-with-y}
        \end{equation}

        We proceed by upper bounding each term of (\ref{eq:break-sum-over-V-with-y}) separately.

        \begin{itemize}
            \item Let $vwz\in\mathcal T(X_0)$.
            For any $j\neq 0$, as $N(v)-X_0-S-\{y\}$ cannot contain two adjacent vertices from the same $X_j$, $v$ can only be adjacent to at most one vertex from each triangle of $X_j$.
            Finally, $v$ is adjacent to exactly three nodes from $X_0\cup S$, leading to
            \[
            \deg_{G'} v + \deg_{G'} w + \deg_{G'} z \le 3\left(\frac{x_1}{3} + \frac{x_2}{3} + \frac{x_3}{3} \right) + 9 = (x_1+x_2+x_3)+9.
            \]
            By repeating the same argument over all $x_i/3$ triangles from $\mathcal T(X_i)$, we have
            \[
            \sum_{vwz\in\mathcal T(X_i)}(\deg_{G'}v+\deg_{G'}w+\deg_{G'}z) \le \frac{x_i}{3}\sum_{j\neq i}x_j + 3x_i.
            \]
            \item As $y$ is not adjacent to any vertex of $S$, we have 
            \[\sum_{v\in S}\deg v = (x_0+x_1+x_2+x_3) + 12 = n+7.\]
            \item For each $i$, $N(y)\cap X_i$ has at most $x_i/3$ vertices, as otherwise by the pigeonhole principle we would have $v,w\in N(y)\cap X_i$ that are adjacent, leading to a triangle $yvw$ sharing an edge with the $K_4$ containing $u_i$, $v$ and $w$.
            Further, $y$ does not have a neighbor in $S$.
            Thus, \[\deg y \le \frac{x_0+x_1+x_2+x_3}3 = \frac{n-5}{3}.\]
        \end{itemize}
        Putting these inequalities together and noting that $3t(G)\le \sum_{v\in V(G)}\deg v$, (\ref{eq:break-sum-over-V-with-y}) gives us
        \[
        \begin{aligned}
        3\lfloor n^2/8\rfloor + 3 \le 3t(G)&\le \frac 23\sum_{i<j}x_ix_j + 3(x_0+x_1+x_2+x_3)+(n+7) + \frac23(n-5)\\
        & = \frac 13(n-5)^2 - \frac13\sum_{i=0}^3x_i^2 + \frac{14n-34}{3}.
        \end{aligned}
        \]
        On the other hand, we note that by the Cauchy-Schwarz inequality, $\sum_{i=0}^3x_i^2 \ge \frac 14(n-5)^2$. 
        Therefore,
        \[
        3\lfloor n^2/8\rfloor + 3 \le \frac 14(n-5)^2+\frac{14n-34}{3} = \frac1{12}(3 n^2 + 26 n - 61),
        \]
        A contradiction to $n\ge 17$. This completes the proof in this case.
        \hfill{$\blacksquare$}

        \paragraph{Case 2: $\sum_{i=0}^4x_i=n-4$.}
        In this case, recall that $G[X_i]$ are disjoint unions of triangles for $0\le i\le 2$, and $X_3$ is a union of triangles and a star on $r\ge 0$ vertices.
        Let us denote this star as $S^\ast = \{c,\ell_1,\ldots, \ell_{r-1}\}$ where $c$ is the center and $\ell_j$ the leaves.
        
        We now continue with the exact same analysis of the neighborhoods of vertices in $G$ as in the previous case.
        For a subset $A\subset V(G)$, let $\mathcal T(A)$ denote the set of triangles in $G[A]$.
        First, we note that
        \begin{equation}
            \label{eq:break-sum-over-V}
            \sum_{v\in V(G)}\deg v = \sum_{i=0}^2\sum_{vwz\in \mathcal T(X_i)}(\deg v+\deg w+\deg z) + \sum_{v\in X_3}\deg v + \sum_{v\in S}\deg v.
        \end{equation}
        Let us now upper bound each term in (\ref{eq:break-sum-over-V}) separately.
        \begin{itemize}
            \item Let $vwz\in \mathcal T(X_0)$.
            Clearly $N(v)-X_0-S$ cannot contain two adjacent vertices from the same $X_j$, $j\neq 0$.
            Therefore, $v$ can only be adjacent with at most one vertex from each triangle of $X_j$ for $j\neq 0$.
            Moreover, $N(v)\cap S^\ast$, $N(w)\cap S^\ast$ and $N(z)\cap S^\ast$ are disjoint, implying 
            \[
            \deg v + \deg w + \deg z \le 3\left(\frac{x_1}{3} + \frac{x_2}{3} + \frac{x_3-r}{3}\right) + r + 9 = (x_1+x_2+x_3) + 9.
            \]
            Similar inequalities hold for each of the $x_i/3$ triangles in $\mathcal T(X_i)$, $0\le i\le 2$.
            In particular, we have
            \[
            \sum_{vwz\in \mathcal T(X_i)}(\deg v+\deg w+\deg z) \le \frac{x_i}3\sum_{j\neq i}x_j + 3x_i.
            \]
            \item Let $v\in X_3$.
            Then, $N(v)-X_3-S$ can have at most one vertex from each triangle of $X_i$.
            Thus,
            \[
            \deg v \le \left\{
            \begin{array}{cl}
            \frac13(x_0+x_1+x_2) + 3, & v\not\in S^\ast,\\
            \frac13(x_0+x_1+x_2) + r, & v = c,\\
            \frac13(x_0+x_1+x_2) + 2, & v \in S^\ast-\{c\}.
            \end{array}
            \right.
            \]
            Thus, if $r\ge 1$,
            \[
            \sum_{v\in X_3}\deg v \le \frac{x_3(x_0+x_1+x_2)}3 + 3(x_3-r) + r + 2(r-1) = \frac{x_3(x_0+x_1+x_2)}3 + 3x_3 - 2,
            \]
            and if $r=0$,
            \[
            \sum_{v\in X_3}\deg v \le \frac{x_3(x_0+x_1+x_2)}3 + 3x_3.
            \]
            We use the latter inequality as it holds for any value of $r$.
            \item Finally, we have
            \[
            \sum_{v\in S}\deg v = (x_0+x_1+x_2+x_3)+12 = n+8.
            \]
        \end{itemize}

        Therefore, (\ref{eq:break-sum-over-V}) along with $3t(G)\le \sum_{v\in V(G)}\deg v$, gives us
        \begin{align}
        3t(G) &\le \frac 23\sum_{i<j}x_ix_j + 3(x_0+x_1+x_2+x_3) + n + 8.\\
        & = \frac13(n-4)^2 - \frac13\sum_{i=0}^3 x_i^2 + 4n - 4 \label{eq:pre-cauchy-schwarz}
        \end{align}

        Observe that by Cauchy-Schwarz, $\sum_{i=0}^3x_i^2\ge \frac14(n-4)^2$.
        Hence, (\ref{eq:pre-cauchy-schwarz}) implies,
        \[
        3t(G)\le \frac14(n-4)^2 + 4n-4 \implies t(G)\le \frac1{12} n(n+8).
        \]
        By $t(G)\ge \lfloor n^2/8\rfloor + 1$, this implies $n\le 14$.
        Note that as $n-4 = \sum_{i=0}^3x_i \ge 9+x_3$, we would have $x_3\le 1$.
        By (\ref{eq:bound-Xi-lower}), this would mean $x_3 = 1$.
        However, this contradicts edge-minimality of $G$, as the edge between $u_3$ and the only vertex of $X_3$ would not be incident to any triangle in $G$, again leading to a contradiction in this case.
        \hfill{$\blacksquare$}

        This completes the proof of the induction step, implying $\ex(n,K_3,\widehat P_4)\le \lfloor n^2/8\rfloor$ for all $n\ge 12$.
        
	\end{proof}

	\section{Concluding Remarks: Uniqueness}
    \label{sec:concl}
    For $n\ge 8$, one may ask whether the lower bound construction of $K_{\lfloor n/2\rfloor, \lceil n/2\rceil}$ with a matching in any of the even parts is unique or not.
    In particular, our proof of Theorem~\ref{thm:mainthm} implies that if the extremal construction contained a $K_4$, then $\lfloor n^2/8\rfloor\le \frac1{12}n(n+8)$.
    This implies $n\le 16$, and indeed, setting $x_i=3$ for every $i$ leads us to an equality case in Case 2.
    
    Our proof therefore gives us the following construction from Figure~\ref{fig:16vertexfig} for $n=16$ consisting entirely of $K_4$-blocks: consider a $K_4$ given by $S=\{u_0,u_1,u_2,u_3\}$.
    For $0\le i\le 3$, let $N(u_i)-S$ consist of the triangles $b_io_ir_i$, where the $b_i$'s are colored blue, $o_i$'s olive and $r_i$'s red.
    Suppose the blue, red and olive vertices each form a $K_4$ (the diagonal edges are omitted in Figure~\ref{fig:16vertexfig} for clarity).
    Clearly each vertex neighborhood has $6$ edges, leading to a total of $16\cdot 6/3=32$ triangles, and hence this graph is a valid extremal configuration for $n=16$.

    \begin{figure}[ht]
    \centering
    \begin{tikzpicture}
        \tikzstyle{vertex}=[circle,fill=black,minimum size=2pt,inner sep=1.5pt]
        \tikzstyle{vertex0}=[circle,fill=red,minimum size=2pt,inner sep=1.5pt]
        \tikzstyle{vertex1}=[circle,fill=olive,minimum size=2pt,inner sep=1.5pt]
        \tikzstyle{vertex2}=[circle,fill=blue,minimum size=2pt,inner sep=1.5pt]
        
        \node[vertex] (u0) at (0,0){};
        \node[vertex] (u1) at (1,0){};
        \node[vertex] (u2) at (1,1){};
        \node[vertex] (u3) at (0,1){};
        \draw (u3)--(u0)--(u1)--(u2)--(u3)--(u1);
        \draw (u2)--(u0);

        \begin{scope}[rotate=180,shift={(-2,0)}]
        \node[vertex0] (u0) at (0,0){};
        \node[vertex] (u1) at (1,0){};
        \node[vertex2] (u2) at (1,1){};
        \node[vertex1] (u3) at (0,1){};
        \draw (u3)--(u0)--(u1)--(u2)--(u3)--(u1);
        \draw (u2)--(u0);
        \end{scope}

        \begin{scope}[rotate=180,shift={(-2,-2)}]
        \node[vertex1] (u0) at (0,0){};
        \node[vertex0] (u1) at (1,0){};
        \node[vertex] (u2) at (1,1){};
        \node[vertex2] (u3) at (0,1){};
        \draw (u3)--(u0)--(u1)--(u2)--(u3)--(u1);
        \draw (u2)--(u0);
        \end{scope}

        \begin{scope}[rotate=180,shift={(0,-2)}]
        \node[vertex2] (u0) at (0,0){};
        \node[vertex1] (u1) at (1,0){};
        \node[vertex0] (u2) at (1,1){};
        \node[vertex] (u3) at (0,1){};
        \draw (u3)--(u0)--(u1)--(u2)--(u3)--(u1);
        \draw (u2)--(u0);
        \end{scope}
        
        \begin{scope}[rotate=180]
        \node[vertex] (u0) at (0,0){};
        \node[vertex2] (u1) at (1,0){};
        \node[vertex1] (u2) at (1,1){};
        \node[vertex0] (u3) at (0,1){};
        \draw (u3)--(u0)--(u1)--(u2)--(u3)--(u1);
        \draw (u2)--(u0);
        \end{scope}

        \draw[line width = 0.3pt, dashed, color=red] (1,2)--(-1,1)--(0,-1)--(2,0)--(1,2);
        \draw[line width = 0.3pt, dashed, color=blue] (2,1)--(0,2)--(-1,0)--(1,-1)--(2,1);
        \draw[line width = 0.3pt, dashed, color=olive] (0.5, 0.5) circle (2.121320344);
    \end{tikzpicture}
    \caption{A $16$-vertex graph with $32$ triangles consisting of only $K_4$-blocks.\label{fig:16vertexfig}}
		\end{figure}

    It seems many extremal constructions are possible for smaller values of $n$ whenever divisibility and structural constraints are satisfied.
    For example, when $n=8$, we enumerate in our repository \cite{Mukherjee_Exact_generalized_Tur_an_2023}  all extremal constructions with $8$ triangles programmatically, and these constructions are comprised of either two edge-disjoint $K_4$'s, or only books.
    However, our proof of Theorem~\ref{thm:mainthm} provides uniqueness of the extremal configuration for $n\ge 17$.
    
	
	\section*{Acknowledgments}
	
	This work was supported by the Center of Innovations for Sustainable Quantum AI (JST Grant Number JPMJPF2221).
	
	
	\bibliographystyle{plain}
	\bibliography{hatP4Free.bbl}

\end{document}